\documentclass[11pt]{amsart}

\newif\ifdraft\draftfalse

\usepackage{amssymb}
\usepackage{amsthm}
\usepackage{eucal}
\usepackage[english]{babel}
\usepackage{nicefrac}
\usepackage{times}
\usepackage{latexsym,amscd,amsmath,amsthm,amssymb}

\makeatletter
\def\@begintheorem#1#2[#3]{%
    \def\naam{#1}
  \deferred@thm@head{\the\thm@headfont \thm@indent
    \@ifempty{#1}{\let\thmname\@gobble}{\let\thmname\@iden}%
    \@ifempty{#2}{\let\thmnumber\@gobble}{\let\thmnumber\@iden}%
    \@ifempty{#3}{\let\thmnote\@gobble}{\let\thmnote\@iden}%
    \thm@swap\swappedhead\thmhead{#1}{#2}{#3}%
    \the\thm@headpunct
    \thmheadnl 
    \hskip\thm@headsep
  }%
  \ignorespaces}
\makeatother

\newcommand{\kantlijndraft}[1]{\ifdraft\hspace{-\lastskip}%
\vadjust{\vspace{-1mm}\smash{\llap{{\tt#1}\hspace{8mm}}}\vspace{1mm}}\fi}

\def\voegToe#1#2#3{\immediate\write1{\string\newlabel{#1}{{#2}{#3}}}}

\newcommand{\thlabel}[1]{\voegToe{#1}{\naam\noexpand~\thetheorem-
}{\thepage}\kantlijndraft{#1}}

\makeatletter
\renewcommand{\label}[1]{\voegToe{#1}{\@currentlabel}{\thepage}\kantlijndraft{#1}}
\makeatother

\newtheorem{theorem}{Theorem}[section]
\newtheorem{lemma}[theorem]{Lemma}
\newtheorem{corollary}[theorem]{Corollary}
\newtheorem{question}[theorem]{Question}

\newtheorem{proposition}[theorem]{Proposition}

\theoremstyle{definition}

\newtheorem{definition}[theorem]{Definition}
\theoremstyle{remark}

\numberwithin{equation}{section}

\newtheorem{claim2}{\sc Claim}



\newcommand{\sse}{\subseteq}                           

\newcommand{\minus}{\backslash}
\newcommand{\Un}{\bigcup}
\newcommand{\un}{\cup}
\newcommand{\Meet}{\bigcap}
\newcommand{\meet}{\cap}
\newcommand{\es}{\varnothing}                          

\newcommand{\closure}[1]{\ensuremath{\overline{#1}}}

\newcommand{\scr}[1]{\ensuremath{\mathcal{#1}}}

\def\cprime{$'$}

\def\sapirovskii{{\v{S}}apirovski{\u\i}}
\def\arhangelskii{Arhangel{\cprime}ski{\u\i}}

\def\juhasz{Juh{\'a}sz}

\begin{document}

\title{Power homogeneous compacta and variations on tightness}

\author{Nathan Carlson}\address{Department of Mathematics, California
Lutheran University, 60 W. Olsen Rd, MC 3750, Thousand Oaks, CA
91360 USA}
\email{ncarlson@callutheran.edu}



\begin{abstract}
The weak tightness $wt(X)$, introduced in \cite{Car2018}, has the property $wt(X)\leq t(X)$. It was shown in \cite{BC2020} that if $X$ is a homogeneous compactum then $|X|\leq 2^{wt(X)\pi\chi(X)}$. We introduce the almost tightness $at(X)$ with the property $wt(X)\leq at(X)\leq t(X)$ and show that if $X$ is a power homogeneous compactum then $|X|\leq 2^{at(X)\pi\chi(X)}$. This improves the result of \arhangelskii, van Mill, and Ridderbos in \cite{AVR2007} that $|X|\leq 2^{t(X)}$ for a power homogeneous compactum $X$ and gives a partial answer to a question in \cite{BC2020}. In addition, if $X$ is a homogeneous Hausdorff space we show that $|X|\leq 2^{pw_cL(X)wt(X)\pi\chi(X)pct(X)}$, improving a result in \cite{BC2020b}. It also extends the result in \cite{BC2020} into the Hausdorff setting. The cardinal invariant $pwL_c(X)$, introduced in \cite{BS2020} by Bella and Spadaro, satisfies $pwL_c(X)\leq L(X)$ and $pwL_c(X)\leq c(X)$. We also show the weight $w(X)$ of a homogeneous space $X$ is bounded in various contexts using $wt(X)$. One such result is that if $X$ is homogeneous and regular then $w(X)\leq 2^{L(X)wt(X)pct(X)}$. This generalizes a result in \cite{BC2020} that if $X$ is a homogeneous compactum then $w(X)\leq 2^{wt(X)}$.
\end{abstract}

\maketitle

\section{Introduction.}

Motivated by results of \juhasz~and van Mill in \cite{JVM2018}, the \emph{weak tightness} $wt(X)$ of a topological space $X$ was introduced in \cite{Car2018} with the property $wt(X)\leq t(X)$ (see Definition~\ref{weakTightness}). This cardinal invariant was investigated further by Bella and the author in \cite{BC2020}, where the following theorem was shown. Recall a space $X$ is \emph{homogeneous} if for all $x,y\in X$ there exists a homeomorphism $h:X\to X$ such that $h(x)=y$. $X$ is \emph{power homogeneous} if there exists a cardinal $\kappa$ such that $X^\kappa$ is homogeneous.

\begin{theorem}[\cite{BC2020}]\label{cpthomog}
Let $X$ be a homogeneous compactum. Then,
\begin{itemize}
\item[(a)] $w(X)\leq 2^{wt(X)}$, and
\item[(b)]$|X|\leq 2^{wt(X)\pi\chi(X)}$. 
\end{itemize}
\end{theorem}

As $\pi\chi(X)\leq t(X)$ for a compactum $X$ (\sapirovskii~\cite{Sap1972}), Theorem~\ref{cpthomog}(b) improved a result of de la Vega \cite{DeLaVega2006} that the cardinality of a homogeneous compactum is at most $2^{t(X)}$. 

It was asked in Question 3.13 in \cite{BC2020} if $2^{wt(X)\pi\chi(X)}$ is a bound for the cardinality of a \emph{power} homogeneous compactum $X$. While this question is still open, we attain a partial answer by introducing a cardinal invariant related to $wt(X)$, the \emph{almost tightness} $at(X)$ of a space $X$ (see Definition~\ref{almostTightness}). The almost tightness, satisfying $wt(X)\leq at(X)\leq t(X)$, resolves technical issues when considering power homogeneous spaces that are not resolved with $wt(X)$ (see section 2). We show in Theorem~\ref{cptPH} that $|X|\leq 2^{at(X)\pi\chi(X)}$ for a power homogeneous compactum $X$. This improves a result of \arhangelskii, van Mill, and Ridderbos \cite{AVR2007} that $|X|\leq 2^{t(X)}$ for such spaces. Theorem ~\ref{cptPH} is related to a result of \juhasz~and van Mill \cite{JVM2018} and also to a result in \cite{Car2018}. See section 3 for the details of these relationships.

In section 4 we generalize Theorem~\ref{cpthomog} into the Hausdorff and regular settings. The first step in this extension process is Theorem~\ref{compactsubset2}, which establishes that if $X$ is Hausdorff and $wt(X)pct(X)\leq\kappa$, then there exists a nonempty compact set $G\sse X$ and $H\sse X$ such that $\chi(G,X)\leq\kappa$, $G\sse\closure{H}$, $|H|\leq 2^\kappa$, and $H$ is $\scr{C}$-saturated. (See \ref{pct} for the definition of $pct(X)$ and \ref{Csat} for the definition of $\scr{C}$-saturated). This represents an extension of Theorem 3.3 in \cite{BC2020} into the Hausdorff setting. That theorem established the existence of $G$ and $H$ with the above properties when $X$ is compact and $wt(X)\leq\kappa$.

We use the cardinal invariant $pwL_c(X)$, first defined in \cite{BS2020} (see Definition~\ref{pwL} below). This invariant has the properties $pwL_c(X)\leq L(X)$ and $pwL_c(X)\leq c(X)$. It was shown in \cite{BS2020} that if $X$ is Hausdorff then $|X|\leq 2^{pwL_c(X)\chi(X)}$, unifying the well-known bounds $2^{L(X)\chi(X)}$ and $2^{c(X)\chi(X)}$ for the cardinality of a Hausdorff space $X$. We show in Theorem~\ref{pbound} that if $X$ is additionally homogeneous then $|X|\leq 2^{pwL_c(X)wt(X)\pi\chi(X)pct(X)}$. This improves Corollary 3.7 in \cite{BC2020b}, which states that $|X|\leq 2^{pwL_c(X)t(X)pct(X)}$ if $X$ is homogeneous and Hausdorff. Theorem~\ref{pbound} also generalizes Theorem~\ref{cpthomog}(b) into the class of Hausdorff spaces. 

We also extend Theorem~\ref{cpthomog}(a) into the regular setting by demonstrating that if $X$ is a homogeneous, regular, Hausdorff space then $w(X)\leq 2^{L(X)wt(X)pct(X)}$. 

By \emph{compactum} we mean a compact, Hausdorff space. For all undefined notions see the books \cite{Engelking} and \cite{Juhasz} by Engelking and \juhasz, respectively. We make no implicit assumption of any separation axiom on a space $X$.

\section{Tools for homogeneous product spaces.}
In this section we modify tools given in~\cite{AVR2007} and in~\cite{Ridderbos2007} for product spaces $X$ and homeomorphisms $h:X\to X$. Variations of these tools will be used to establish a result for power homogeneous spaces, Corollary~\ref{GPH}, that will be needed in the proof of Theorem~\ref{cptPH}. 

For this section we fix the following notation. Let $X=\prod\{X_i:i\in I\}$ be endowed with the product topology, where $X_i$ is a space for all $i\in I$. Let $H=\prod\{H_i:i\in I\}$ where $H_i\sse X_i$, and for all $i\in I$, fix $s_i\in H_i$. If $A\sse I$, let
$$H(A)=\{x\in H:x_i=s_i\textup{ except for finitely many }i\in A\}.$$
When $K\sse X$, let $K_A=\pi_A[K]$, where $\pi_A:X\to\prod\{X_i:i\in A\}$ is the projection map. If $x\in X$ and $\pi_j: X\to X_j$ is the projection for some $j\in I$, then we set $x_j=\pi_j(x)$.

The following is Corollary 2.9 in~\cite{AVR2007}, re-worded to fit our notation.

\begin{theorem}[Corollary 2.9~\cite{AVR2007}]\label{AVR}
Suppose that $X$ is homogeneous and for all $i\in I$ there exists a non-empty closed $G_\kappa$-set $G_i$ and a set $H_i\in[X_i]^{\leq\kappa}$ such that $G\sse\closure{H_i}$. If for some $j\in I$ we have $\pi\chi(X_j)\leq\kappa$, then there exists a cover $\scr{G}$ of $X_j$ consisting of closed $G_\kappa$-sets such that for all $G\in\scr{G}$ there exists $H_G\in[X_j]^{\leq\kappa}$ such that $G\sse\closure{H_G}$.
\end{theorem}

A set $G$ is a $G^c_\kappa$-\emph{set} of a space $Y$ if there exists a family of open sets $\scr{U}$ in $Y$ such that $|\scr{U}|\leq\kappa$ and $G=\Meet\scr{U}=\Meet_{U\in\scr{U}}\closure{U}$. Note that a $G^c_\kappa$-set is a special type of closed $G_\kappa$-set. Observe that a compact $G_\kappa$-set in a Hausdorff space $Y$ is a $G^c_\kappa$-set. Our aim in this section is to ensure that if each $G_i$ in the Theorem~\ref{AVR} is a $G^c_\kappa$-set in $X_i$ then each member of the cover $\scr{G}$ is a $G^c_\kappa$-set in $X_j$. We accomplish this in Corollary~\ref{PH1}.

The following are two results from~\cite{AVR2007}. The first gives straightforward properties of the set $H(A)$ for $A\in[I]^{\leq\kappa}$. The second gives conditions for the existence of a set $A\in[I]^{\leq\kappa}$ with certain properties, given a homeomorphism $h:X\to X$.

\begin{lemma}[Lemma 2.7~\cite{AVR2007}]\label{S(A)}
If $A\sse I$ is such that $|A|\leq\kappa$ for some cardinal $\kappa$, and for all $i\in A$, $|H_i|\leq\kappa$, then $|H(A)|\leq \kappa$ and $H(A)_A$ is dense in $H_A$. Futhermore, if $A\sse B$ then $H(A)\sse H(B)$. If $(A_n)$ is an increasing sequence of infinite subsets of $I$ and $A=\Un_{n=1}^{\infty} A_n$, then $H(A)=\Un_{n=1}^\infty H(A_n)$.
\end{lemma}

\begin{theorem}[Theorem 2.8~\cite{AVR2007}]\label{getA2kappa}
Suppose there exists a cardinal $\kappa$ such that $|H_i|\leq\kappa$ for all $i\in I$. Suppose that for some $j\in I$, $\pi\chi(X_j)\leq\kappa$. If $h:X\to X$ is a homeomorphism and $B\sse I$ such that $|B|\leq \kappa$, then there is a set $A\sse I$ such that $|A|\leq\kappa$, $B\sse A$, and for all $s\in H(A)$ and for all $x\in X$, if $x_A=s_A$ then $h(s)_j=h(x)_j$.
\end{theorem}

We obtain the following new result, which is an adjustment to Corollary 2.9 in \cite{AVR2007}.

\begin{corollary}\label{PH1}
Suppose $X$ is homogeneous and Hausdorff and suppose for all $i\in I$ there exist $e_i\in X_i$, a $G^c_\kappa$-set $G_i\sse X_i$, and $H_i\sse X_i$ such that $|H_i|\leq\kappa$ and $e_i\sse G_i\sse\closure{H_i}$. If for some $j\in I$ we have $\pi\chi(X_j)\leq\kappa$, then for all $z\in X_j$ there exists a $G^c_\kappa$-set $F$ in $X_j$ and $T\sse X_j$ such that $|T|\leq\kappa$ and $z\in F\sse\closure{T}$. 
\end{corollary}

\begin{proof}
Let $G=\prod\{G_i:i\in I\}$ and let $z\in X_j$ be arbitrary. Pick $w\in X$ such that $w_j=z$. As $X$ is homogeneous, there exists a homeomorphism $h:X\to X$ such that $h(e)=w$, where $e$ is the point of $X$ whose $i^{th}$ coordinate is $e_i$. By Theorem~\ref{getA2kappa} and Lemma~\ref{S(A)} we obtain $A\sse I$ such that $|A|\leq\kappa$ and 
\begin{enumerate}
\item $|H(A)|\leq\kappa$,
\item $e_A\in G_A\sse\closure{H(A)_A}$, and
\item for all $s\in H(A)$ and all $x\in X$ if $x_A=s_A$ then $h(x)_j=h(s)_j$.
\end{enumerate}
Now, for all $s\in H(A)$ the set $\pi_jh\pi_A^{-1}(s_A)$ consists of the single point $h(s)_j$. Let $T=\{h(s)_j:s\in H(A)\}$ and note $T\in [X_j]^{\leq\kappa}$. Since $G_A\sse\closure{H(A)_A}$ and $\pi_A$ is open, we have $\pi_A^{-1}[G_A]\sse\closure{\pi_A^{-1}[H(A)_A]}$. Thus
$$\pi_jh[\pi_A^{-1}[G_A]]\sse\closure{\pi_jh\pi_A^{-1}[H(A)_A]}=\closure{T}.$$

We show that $\pi_A[G_A]$ is a $G^c_\kappa$-set in $X$. As each $G_i$ is a $G^c_\kappa$-set in $X_i$, for all $i\in I$ there exists a family $\{U(i,\alpha):\alpha<\kappa\}$ of open sets in $X_i$ such that $G_i=\Meet_{\alpha<\kappa}U(i,\alpha)=\Meet_{\alpha<\kappa}\closure{U(i,\alpha)}$. Now,
\begin{align}
\pi_A^{-1}[G_A]&=\prod_{i\in A}G_i\times\prod_{i\notin A}X_i=\prod_{i\in A}(\Meet_{\alpha<\kappa}U(i,\alpha))\times\prod_{i\notin A}X_i\\
&=\Meet_{\alpha<\kappa}(\prod_{i\in A}U(i,\alpha)\times\prod_{i\notin A}X_i)=\Meet_{\alpha<\kappa}\Meet_{i\in A}(U(i,\alpha)\times\prod_{l\neq i}X_l).\notag
\end{align}

Therefore $\pi_A^{-1}[G_A]$ is the intersection of $\kappa$-many open sets in $X$, as $|A|\leq\kappa$. Similarly, it can be shown that $\pi_A^{-1}[G_A]=\Meet_{\alpha<\kappa}\Meet_{i\in A}\closure{U(i,\alpha)\times\prod_{l\neq i}X_l}$, making $\pi_A^{-1}[G_A]$ a $G^c_\kappa$-set in $X$. It follows that $h[\pi_A^{-1}[G_A]]$ is also a $G^c_\kappa$-set in $X$ as $h$ is a homeomorphism.

Let $B=I\minus\{j\}$ and $Y=\{y\in X:y_B=w_B\}$. Then $\left.\pi_j\right|_Y:Y\to X_j$ is a homeomorphism. Now let $K=Y\meet h[\pi_A^{-1}[G_A]]$. We show $K$ is a $G^c_\kappa$-set of $Y$. There exists a family $\scr{U}$ of open sets of $X$ such that $|\scr{U}|\leq\kappa$ and $h[\pi_A^{-1}[G_A]]=\Meet\scr{U}=\Meet_{U\in\scr{U}}\closure{U}$. Then,
$$K=Y\meet\Meet\scr{U}=\Meet_{U\in\scr{U}}(U\meet Y)\sse\Meet_{U\in\scr{U}}(\closure{U\meet Y}\meet Y)\sse\Meet_{U\in\scr{U}}(\closure{U}\meet Y)= K,$$
and therefore $K=\Meet_{U\in\scr{U}}(U\meet Y)=\Meet_{U\in\scr{U}}(\closure{U\meet Y}\meet Y)=\Meet_{U\in\scr{U}}cl_Y(U\meet Y)$.

This shows $K$ is a $G^c_\kappa$-set of $Y$. Let $F=\pi_j[K]$. Then $F$ is a $G^c_\kappa$-set of $X_j$ as $\left.\pi_j\right|_Y:Y\to X_j$ is a homeomorphism. We have,
$$z=w_j\in F=\pi_j[K]\sse\pi_jh[\pi_A^{-1}[G_A]]\sse\closure{T}.$$
Thus $z\in F\sse\closure{T}$. This completes the proof.
\end{proof}

For our last corollary in this section we dispense with the convention that $X$ is a product space $\prod\{X_i:i\in I\}$ and that $H=\prod\{H_i:i\in I\}$. It follows immediately from Corollary~\ref{PH1}. Compare this to Theorem~\ref{AVR}.

\begin{corollary}\label{GPH}
Let $X$ be a power homogeneous Hausdorff space. Suppose there exists a nonempty $G^c_\kappa$-set $G$ and a set $H\in[X]^{\leq\kappa}$ such that $G\sse\closure{H}$. Then there exists a cover $\scr{G}$ of $X$ consisting of $G^c_\kappa$-sets such that for all $G\in\scr{G}$ there exists $H_G\in[X]^{\leq\kappa}$ such that $G\sse\closure{H_G}$.
\end{corollary}

Given Theorem~\ref{compactsubset} in the next section, one might wonder whether one can instead suppose $|H|\leq 2^\kappa$ in the above corollary and conclude that each $H_G$ has cardinality at most $2^\kappa$ while keeping each $G\in\scr{G}$ a $G^c_\kappa$-set. However, an analysis of the subtleties of the proof of Corollary~\ref{PH1} does not allow this adjustment. If each $|H_i|\leq 2^\kappa$ in Corollary~\ref{PH1} then in fact the $A$ we obtain from Theorem~\ref{getA2kappa} will have cardinality at most $2^\kappa$. In that case, one sees from equation (2.1) in the proof of Corollary~\ref{PH1} that $\pi^{-1}[G_A]$ is a $G^c_{2^\kappa}$-set, not necessarily a $G^c_\kappa$-set. It is essentially for this reason that the cardinal invariant $at(X)$ is introduced, whereby Theorem~\ref{atcompactsubset} can be used in place of Theorem~\ref{compactsubset} in conjunction with the above corollary. This combination is used in the proof of Theorem~\ref{cptPH}, which gives a cardinality bound for any power homogeneous compactum.

\section{The cardinality of a power homogeneous compactum.}

We begin this section by defining the weak tightness $wt(X)$ of a space $X$ and related notions. Given a cardinal $\kappa$, a space $X$, and $A\sse X$, the $\kappa$-\emph{closure} of $A$ is defined as $cl_\kappa A=\Un_{B\in[A]^{\leq\kappa}}\closure{B}$. The definition of $\kappa$-closure was given in~\cite{CPR2012} and was used in~\cite{JVM2018}. Observe that $cl_\kappa(cl_\kappa A)=cl_\kappa A\sse\closure{A}$ and if $cl_\kappa A=X$ then $A$ is dense in $X$. 

\begin{definition}[\cite{Car2018}]\label{weakTightness}
Let $X$ be a space. The \emph{weak tightness} $wt(X)$ of $X$ is defined as the least infinite cardinal $\kappa$ for which there is a cover $\scr{C}$ of $X$ such that $|\scr{C}|\leq 2^\kappa$ and for all $C\in\scr{C}$, $t(C)\leq\kappa$ and $X=cl_{2^\kappa}C$. We say that $X$ is \emph{weakly countably tight} if $wt(X)=\omega$. 
\end{definition}

\begin{definition}\label{almostTightness}
Let $X$ be a space. The \emph{almost tightness} $at(X)$ of $X$ is defined as the least infinite cardinal $\kappa$ for which there is a cover $\scr{C}$ of $X$ such that $|\scr{C}|\leq\kappa$ and for all $C\in\scr{C}$, $t(C)\leq\kappa$ and $X=cl_{\kappa}C$. We say that $X$ is \emph{almost countably tight} if $at(X)=\omega$. 
\end{definition}

It is clear that $wt(X)\leq at(X)\leq t(X)$. See Example 2.3 in \cite{BC2020} for a consistent example of a compact space $X$ for which $wt(X)=\aleph_0<\aleph_1=at(X)=t(X)$. 

We will see that while $wt(X)$ is used in cardinal inequalities on homogeneous spaces (Theorems~\ref{cpthomog}, \ref{pbound}, \ref{w(X)}), $at(X)$ is more suited for cardinal inequalities on \emph{power} homogeneous spaces (Theorem~\ref{cptPH}). This is because $at(X)$ ``interacts'' well with Corollary~\ref{GPH}. (See comment after that corollary).

In \cite{JVM2018}, \juhasz~and van Mill introduced the notion of a $\scr{C}$-saturated subset of a space $X$. 
\begin{definition}[\cite{JVM2018}]\label{Csat}
Given a cover $\scr{C}$ of $X$, a subset $A\sse X$ is $\scr{C}$-\emph{saturated} if $A\meet C$ is dense in $A$ for every $C\in\scr{C}$. 
\end{definition}
It is clear that the union of $\scr{C}$-saturated subsets is $\scr{C}$-saturated. The following was mentioned in the proof of Lemma 3.2 in \cite{JVM2018} and proved in Proposition 3.2 in \cite{Car2018}. 

\begin{proposition}\label{saturated}
Let $X$ be a space, $wt(X)=\kappa$, and let $\scr{C}$ be a cover witnessing that $wt(X)=\kappa$. Then for all $x\in X$ there exists $S(x)\in[X]^{\leq 2^\kappa}$ such that $x\in S(x)$ and $S(x)$ is $\scr{C}$-saturated. 
\end{proposition}

Whenever $wt(X)=\kappa$ and $x\in X$, we fix $S(x)$ as obtained in Proposition~\ref{saturated}. If $A\sse X$, then we set $S(A)=\Un_{x\in A}S(x)$. 

The following is an adaptation of Proposition~\ref{saturated} for the invariant $at(X)$. Observe that while $S(x)\in[X]^{\leq 2^\kappa}$, in fact $T(x)\in[X]^{\leq\kappa}$. It is this crucial difference that ultimately allows for the cardinal inequality in Theorem~\ref{cptPH} to work for power homogeneous compacta using $at(X)$ rather than $wt(X)$.

\begin{proposition}\label{atsaturated}
Let $X$ be a space, $at(X)=\kappa$, and let $\scr{C}$ be a cover witnessing that $at(X)=\kappa$. Then for all $x\in X$ there exists $T(x)\in[X]^{\leq\kappa}$ such that $x\in T(x)$ and $T(x)$ is $\scr{C}$-saturated. 
\end{proposition}

\begin{proof}
Fix $x\in X$. We inductively construct a sequence of subsets $A_n\sse X$ such that for all $n$, $|A_n|\leq\kappa$. For all $C\in\scr{C}$ there exists $A_C\in[C]^{\leq\kappa}$ such that $x\in\closure{A_C}$. Let $A_1=\Un_{C\in\scr{C}}A_C$ and note that $|A_1|\leq\kappa$.  Suppose that $A_n$ has been constructed for $n<\omega$. For all $y\in A_n$ and for all $C\in\scr{C}$ there exists $A(C,y,n)\in [C]^{\leq\kappa}$ such that $y\in\closure{A(C,y,n)}$. Let $A_{n+1}= \Un\{A(C,y,n):y\in A_n, C\in\scr{C}\}$ and note that $|A_{n+1}|\leq\kappa$.

Define $T(x)=\{x\}\un\Un_{n<\omega}A_n$. Then $x\in T(x)$ and $|T(x)|\leq\kappa$. We show that $T(x)$ is $\scr{C}$-saturated. Let $C\in\scr{C}$. We show $T(x)\sse\closure{C\meet T(x)}$. By above, $x\in\closure{A_C}$ and $A_C\sse C\meet A_1\sse C\meet T(x)$. Thus $x\in\closure{C\meet T(x)}$. Now let $n\leq\omega$ and $y\in A_n$. By above, $y\in\closure{A(C,y,n)}$ and $A(C,y,n)\sse C\meet A_{n+1}\sse C\meet T(x)$. Thus $y\in\closure{C\meet T(x)}$. This shows that $T(x)\sse\closure{C\meet T(x)}$ and that $T(x)$ is $\scr{C}$-saturated.
\end{proof}

Whenever $at(X)=\kappa$ and $x\in X$, we fix $T(x)$ as obtained in Proposition~\ref{atsaturated}. If $A\sse X$, then we set $T(A)=\Un_{x\in A}T(x)$. 

\begin{lemma}[Lemma 2.7 in \cite{Car2018}]\label{closureunion}
Let $X$ be a space, $\kappa$ a cardinal such that $wt(X)\leq\kappa$, and $\scr{C}$ be a cover of $X$ witnessing that $wt(X)\leq\kappa$. If $\scr{B}$ is an increasing chain of $\kappa^+$-many $\scr{C}$-saturated subsets of $X$, then
$$\closure{\Un\scr{B}}=\Un_{B\in\scr{B}}\closure{B}.$$
\end{lemma}

The notion of an $S$-\emph{free sequence} was used heavily in connection with the weak tightness $wt(X)$ in \cite{BC2020}. Recall that for a cardinal $\kappa$ and a space $X$, a set $\{x_\alpha:\alpha<\kappa\}\sse X$ is a \emph{free sequence} if
$\closure{\{x_\beta:\beta<\alpha\}}\meet\closure{\{x_\beta:\alpha\leq\beta<\kappa\}}=\es$ for all $\alpha<\kappa$. We define
the notion of an $S$-free sequence below. Note that every $S$-free sequence is a free sequence.

\begin{definition}[\cite{BC2020}]\label{sfree}
Let $wt(X)=\kappa $. A set $\{x_\alpha:\alpha<\lambda \}$ is an \emph{S-free sequence} if
$\closure{S(\{x_\beta:\beta<\alpha\})}\meet\closure{\{x_\beta:\alpha\leq\beta<\lambda \}}=\es$ for all $\alpha<\lambda $.
\end{definition}

It was shown in \cite{BC2020} that if $X$ is compact and $wt(X)=\kappa$, then $X$ contains no $S$-free sequences of length $\kappa^+$. The following is related and is a key result in generalizing Theorem~\ref{cpthomog} to a non-compact setting in the next section.

\begin{proposition}\label{noSfree}
Let $X$ be a space, $K\sse X$ be a compact set, and $wt(X)=\kappa$. Then $K$ contains no $S$-free sequence of length $\kappa^+$. 
\end{proposition}

\begin{proof}
By way of contradiction suppose $A=\{x_\alpha:\alpha<\kappa^+\}$ is an $S$-free sequence contained in $K$. As $K$ is compact, $A$ has a complete accumulation point $x$ in $K$. Thus $|U\meet K\meet A|=\kappa^+$ for all open sets $U$ of $X$ containing $x$. Then, $x\in cl_KA\sse\closure{A}\sse\closure{S(A)}=\Un_{\alpha<\kappa^+}\closure{S(\{x_\beta:\beta<\alpha\})}$, where the latter equality follows from Lemma~\ref{closureunion} and the fact that $\kappa=wt(X)$. Thus there exists $\alpha<\kappa^+$ such that $x\in\closure{S(\{x_\beta:\beta<\alpha\})}$. It follows that $x\in K\minus\closure{\{x_\beta:\alpha\leq\beta<\kappa^+\}}$, an open set in $K$. This implies that $\kappa^+=|A\minus\closure{\{x_\beta:\alpha\leq\beta<\kappa^+\}}|\leq |\{x_\beta:\beta<\alpha\}|=\kappa$, a contradiction.
\end{proof}

We introduce the notion of a $T$-free sequence for use with the invariant $at(X)$. 

\begin{definition}\label{tfree}
Let $at(X)=\kappa $. A set $\{x_\alpha:\alpha<\lambda \}$ is an \emph{T-free sequence} if
$\closure{T(\{x_\beta:\beta<\alpha\})}\meet\closure{\{x_\beta:\alpha\leq\beta<\lambda \}}=\es$ for all $\alpha<\lambda $.
\end{definition}

It can also be shown that if $K$ is a compact subset of a space $X$ and $at(X)=\kappa$, then $K$ contains no $T$-free sequence of length $\kappa^+$. The proof is similar to that of Proposition~\ref{noSfree}.

\begin{proposition}\label{noTfree}
Let $X$ be a space, $K\sse X$ where $K$ is compact, and $at(X)=\kappa$. Then $K$ contains no $T$-free sequence of length $\kappa^+$. 
\end{proposition}

\arhangelskii~showed in \cite{Arh1978} that if $X$ is a compactum and $t(X)\leq\kappa$, then there exists a nonempty compact set $G$ and a set $H$ such that $G\sse\closure{H}$, $|H|\leq\kappa$ and $\chi(G,X)\leq\kappa$. \juhasz~and van Mill \cite{JVM2018} improved this in the countable case by showing that if $X$ is a compactum that is the countable union of countably tight subspaces then there exists a nonempty closed $G_\delta$-set contained in the closure of a countable set. Theorem 3.3 in \cite{BC2020} gave a variation of these results, which states that if $X$ is a compactum and $wt(X)\leq\kappa$, then there exists a nonempty compact set $G$ and a $\scr{C}$-saturated set $H$ such that $G\sse\closure{H}$, $|H|\leq 2^\kappa$, and $\chi(G,X)\leq\kappa$. We prove an extension of this result into the Hausdorff setting.

\begin{theorem}\label{compactsubset}
Let $X$ be a Hausdorff space, $\kappa=wt(X)$, and $K$ a nonempty compact subset of $X$. Then there exists a nonempty closed set $G\sse K$ and a set $H\sse X$ such that $|H|\leq 2^\kappa$, $G\sse\closure{H}$, and $\chi(G,K)\leq\kappa$. In addition, $H$ is $\scr{C}$-saturated for any cover $\scr{C}$ witnessing $\kappa=wt(X)$.
\end{theorem}

\begin{proof}
Let $\scr{C}$ be a cover of $X$ witnessing that $\kappa=wt(X)$. Set $\scr{F}=\{F\sse K:F\neq\es, F\textup{ closed },\chi(F,K)\leq\kappa\}$. Suppose by way of contradiction that for all $F\in\scr{F}$ and all $\scr{C}$-saturated $H\in [X]^{\leq 2^\kappa}$ we have $F\minus\closure{H}\neq\es$. We construct a decreasing sequence $\scr{F}^\prime=\{F_\alpha:\alpha<\kappa^+\}\sse \scr{F}$ and an $S$-free sequence $A=\{x_\alpha:\alpha<\kappa^+\}\sse K$.

For $\alpha<\kappa^+$, suppose $x_\beta\in K$ and $F_\beta$ have been defined for all $\beta<\alpha<\kappa^+$, where $F_{\beta^{\prime\prime}}\sse F_{\beta^\prime}$ for all $\beta^\prime<\beta^{\prime\prime}<\alpha$. Define $A_\alpha=\{x_\beta:\beta<\alpha\}$ and $G_\alpha=\Meet\{F_\beta:\beta<\alpha\}$. As $K$ is compact, $G_\alpha\neq\es$ and $G_\alpha\in\scr{F}$. As $|A_\alpha|\leq\kappa$, we have $|S(A_\alpha)|\leq\kappa\cdot 2^\kappa=2^\kappa$. Thus, $G_\alpha\minus\closure{S(A_\alpha)}\neq\es$ as $S(A_\alpha)$ is $\scr{C}$-saturated.

Let $x_\alpha\in G_\alpha\minus\closure{S(A_\alpha)}$. As $\chi(G_\alpha,K)\leq\kappa$, we have that $G_\alpha$ is a $G_\kappa$-set in $K$. There exists a family $\scr{U}$ of open sets in $X$ such that $|\scr{U}|\leq\kappa$ and $G_\alpha=\Meet\{U\meet K:U\in\scr{U}\}$. As $K$ is regular, for all $U\in\scr{U}$ there exists a closed $G_\delta$-set in $K$ such that $x_\alpha\in G_U\sse(U\meet K)\minus\closure{S(A_\alpha)}$, as $(U\meet K)\minus\closure{S(A_\alpha)}$ is open in $K$.

Set $F_\alpha=\Meet\{G_U:U\in\scr{U}\}$. Note $x_\alpha\in F_\alpha\sse G_\alpha\minus\closure{S(A_\alpha)}$ and that $F_\alpha$ is a nonempty closed $G_\kappa$-set in $K$. As $K$ is compact, $\chi(F_\alpha,K)\leq\kappa$ and $F_\alpha\in\scr{F}$. The choice of $x_\alpha$ and $F_\alpha$ completes the construction of the sequences $\scr{F}^\prime$ and $A$.

We show now that $A$ is an $S$-free sequence. Let $\alpha<\kappa^+$. Note $\{x_\beta:\alpha\leq\beta<\kappa^+\}\sse F_\alpha$ as $\scr{F}^\prime$ is a decreasing sequence. Then $\closure{\{x_\beta:\alpha\leq\beta<\kappa^+\}}\sse F_\alpha$. Also, $\{x_\beta:\beta<\alpha\}=A_\alpha$, so $\closure{S(\{x_\beta:\beta<\alpha\})}=\closure{S(A_\alpha)}\sse X\minus F_\alpha$. It follows that $\closure{S(\{x_\beta:\beta<\alpha\})}\meet\closure{\{x_\beta:\alpha\leq\beta<\kappa^+\}}=\es$ and $A$ is an $S$-free sequence contained in $K$ of length $\kappa^+$. This contradicts Proposition~\ref{noSfree}. Thus there exists $G$ and $H$ as required.
\end{proof}

We have a similar result for when $\kappa=at(X)$. The proof is identical, except that instances of $2^\kappa$ are replaced by $\kappa$, a $T$-free sequence is constructed, and Proposition~\ref{noTfree} is used instead of Proposition~\ref{noSfree}.

\begin{theorem}\label{atcompactsubset}
Let $X$ be a Hausdorff space, $\kappa=at(X)$, and $K$ a nonempty compact subset of $X$. Then there exists a nonempty closed set $G\sse K$ and a set $H\sse X$ such that $|H|\leq\kappa$, $G\sse\closure{H}$, and $\chi(G,K)\leq\kappa$. In addition, $H$ is $\scr{C}$-saturated in any cover $\scr{C}$ witnessing that $at(X)=\kappa$.
\end{theorem}

\begin{lemma}[Lemma 2.4.3 in \cite{Ridderbos2007}]\label{char}
Let $X$ be a space. Suppose $F$ is a compact subset of $Y$ and $Y$ is a compact subset of $X$. Then $\chi(F,X)\leq\chi(F,Y)\chi(Y,X)$.
\end{lemma}

\begin{definition}\label{pct}
The \emph{point-wise compactness type} $pct(X)$ of a space $X$ is the least infinite cardinal $\kappa$ such that $X$ can be covered by compact sets $K$ such that $\chi(K,X)\leq\kappa$.
\end{definition}

Note that if $X$ is compact then $pct(X)=\omega$. Furthermore, it can be shown that $\chi(X)=\psi(X)pct(X)$.

\begin{theorem}\label{compactsubset2}
If $X$ is a Hausdorff space and $wt(X)pct(X)\leq\kappa$, then there exists a nonempty compact set $G\sse X$ and $H\sse X$ such that $\chi(G,X)\leq\kappa$, $G\sse\closure{H}$, $|H|\leq 2^\kappa$, and $H$ is $\scr{C}$-saturated in any cover $\scr{C}$ witnessing $wt(X)\leq\kappa$.
\end{theorem}

\begin{proof}
Since $pct(X)\leq\kappa$, there exists a nonempty compact set $K$ such that $\chi(K,X)\leq\kappa$. (In fact, there is a cover of such sets). By Theorem~\ref{compactsubset} there exists a nonempty compact set $G\sse K$ and a $\scr{C}$-saturated set $H\in[X]^{\leq 2^\kappa}$ such that $G\sse\closure{H}$ and $\chi(G,K)\leq\kappa$. Now, by Lemma~\ref{char}, $\chi(G,X)\leq\chi(G,K)\cdot\chi(K,X)=\kappa\cdot\kappa=\kappa$, as $G$ and $K$ are compact.
\end{proof}

We have a related result for when $at(X)pct(X)\leq\kappa$. Notice that $|H|\leq 2^\kappa$ in Theorem~\ref{compactsubset2} and $|H|\leq\kappa$ in Theorem~\ref{atcompactsubset2}. It's proof is similar to the above and uses Theorem~\ref{atcompactsubset}.

\begin{theorem}\label{atcompactsubset2}
If $X$ is a Hausdorff space and $at(X)pct(X)\leq\kappa$, then there exists a nonempty compact set $G\sse X$ and $H\sse X$ such that $\chi(G,X)\leq\kappa$, $G\sse\closure{H}$, $|H|\leq\kappa$, and $H$ is $\scr{C}$-saturated in any cover $\scr{C}$ witnessing $at(X)\leq\kappa$.
\end{theorem}

The following theorem was established in \cite{Car2018}. It represents an improvement and generalization of the result of Pytkeev~\cite{Pyt1985} that $L(X_\kappa)\leq 2^{t(X)\cdot\kappa}$ for compact spaces $X$, where $X_\kappa$ is the $G_\kappa$-\emph{modification} of a space $X$. $X_\kappa$ is the space formed on the underlying set $X$ where the $G_\kappa$-sets of $X$ form a basis. Given a space $X$, $X^c_\kappa$ represents the $G^c_\kappa$-\emph{modification} of $X$, the space formed on $X$ where the $G^c_\kappa$-sets form a basis. 

\begin{theorem}[Main Theorem in \cite{Car2018}]\label{L(X)}
For any space $X$ and cardinal $\kappa$, $L(X^c_\kappa)\leq 2^{L(X)wt(X)\cdot\kappa}$.
\end{theorem}

\begin{lemma}[Theorem 3.4, \cite{Rid2006}]\label{ridPH}
If $X$ is a power homogeneous Hausdorff space then $|X|\leq d(X)^{\pi\chi(X)}$.
\end{lemma}

We now have all the ingredients to prove the central theorem in this section. It ties together the main result in the previous section with results in this section.

\begin{theorem}\label{cptPH}
If $X$ is a power homogeneous compactum then $|X|\leq 2^{at(X)\pi\chi(X)}$.
\end{theorem}

\begin{proof}
Let $\kappa=at(X)\pi\chi(X)$. Let $\scr{C}$ be a cover witnessing that $at(X)\leq\kappa$. By Theorem~\ref{atcompactsubset2} there exists a nonempty compact set $G$ and a set $H\in[X]^{\leq\kappa}$ such that $H$ is $\scr{C}$-saturated, $G\sse\closure{H}$, and $\chi(G,X)\leq\kappa$. As $G$ is compact, it follows that $G$ is a $G^c_\kappa$-set. Since $X$ is power homogeneous and $\pi\chi(X)\leq\kappa$, by Corollary~\ref{GPH} there exists a cover $\scr{G}$ of $X$ of nonempty $G^c_\kappa$-sets such that for each $G\in\scr{G}$ there exists $H_G\in[X]^{\leq\kappa}$ such that $G\sse\closure{H_G}$. As $wt(X)\leq at(X)$, Theorem~\ref{L(X)} allows us to assume that $|\scr{G}|\leq 2^\kappa$. Then $X=\Un\scr{G}\sse\Un_{G\in\scr{G}}\closure{H_G}\sse\closure{\Un_{G\in\scr{G}}H_G}$, showing $D=\Un_{G\in\scr{G}}H_G$ is dense. As $|D|\leq 2^\kappa\cdot\kappa=2^\kappa$, we have $d(X)\leq 2^\kappa$. By Lemma~\ref{ridPH}, it follows that $|X|\leq d(X)^{\pi\chi(X)}\leq (2^{\kappa})^\kappa=2^\kappa$, completing the proof.
\end{proof}

Theorem~\ref{cptPH} improves the result in \cite{AVR2007} that the cardinality of a power homogeneous compactum is at most $2^{t(X)}$, as $\pi\chi(X)\leq t(X)$ for a compactum $X$ and $at(X)\leq t(X)$ for any space. It also gives a partial answer to Question 3.13 in \cite{BC2020}, which asks if $|X|\leq 2^{wt(X)\pi\chi(X)}$ for a power homogeneous compactum.  

Notice that the fact that $\pi\chi(X)\leq\kappa$ is used twice in the above proof: once in the use of Corollary~\ref{GPH} and again in the use of Lemma~\ref{ridPH}. Thus it would seem to be difficult to remove this invariant from the cardinality bound $2^{at(X)\pi\chi(X)}$ for power homogeneous compacta.

As Theorem~\ref{atcompactsubset2} applies if $at(X)pct(X)\leq\kappa$ and Theorem~\ref{L(X)} applies if $L(X)at(X)\leq\kappa$, 
observe that the previous proof in fact demonstrates the following more general result. 

\begin{theorem}\label{PHbound}
Let $X$ be a power homogeneous Hausdorff space. Then $|X|\leq 2^{L(X)at(X)\pi\chi(X)pct(X)}$.
\end{theorem}

A result in \cite{Car2018} (given in Theorem~\ref{previousPH} below) is a consequence of Theorem~\ref{cptPH}. The next several results establish the reasons for this. Recall that a space is $\sigma$-CT if it is a countable union of countably tight subspaces. Observe that an almost countably tight space is $\sigma$-CT. The following was shown by \juhasz~and van Mill in \cite{JVM2018}.

\begin{theorem}[Lemma 2.4 in \cite{JVM2018}]\label{closedsubset}
If $X$ is a $\sigma$-CT compactum then every nonempty closed subspace $F$ has a point $x\in F$ such that $\pi\chi(x,F)\leq\omega$.
\end{theorem}

\begin{theorem}[Lemma 4.2 in \cite{Car2018}]\label{closedsubset2}
Let $X$ be a power homogeneous compactum and let $\kappa$ be an infinite cardinal. If for all nonempty closed subsets $F\sse X$ there exists $x\in F$ such that $\pi\chi(x,F)\leq\kappa$, then $\pi\chi(X)\leq\kappa$.
\end{theorem}

The following is a direct corollary of the previous two theorems.

\begin{corollary}\label{pichisigma}
Let $X$ be a $\sigma$-CT power homogeneous compactum. Then $\pi\chi(X)\leq\omega$. 
\end{corollary}

An interesting consequence of Corollary~\ref{pichisigma} is the next result. It follows from the fact that $|X|\leq 2^{c(X)\pi\chi(X)}$ for any power homogeneous Hausdorff space $X$ \cite{CR2008}.

\begin{corollary}
Let $X$ be a $\sigma$-CT power homogeneous compactum. Then $|X|\leq 2^{c(X)}$.
\end{corollary}

If there is a cover $\scr{C}$ of a space $X$ and a cardinal $\kappa$ for which $|\scr{C}|\leq\kappa$ and for all $C\in\scr{C}$, $t(C)\leq\kappa$ and $X=cl_\kappa C$ then, by definition, $at(X)\leq\kappa$. As the next proposition shows, if there is an additional cardinal restriction on the $\pi$-character $\pi\chi(X)$ then the condition ``$X=cl_\kappa C$ for all $C\in\scr{C}$ can be relaxed to ``$C$ is dense in $X$ for all $C\in\scr{C}$'' to still guarantee that $at(X)\leq\kappa$.

\begin{proposition}\label{atcondition}
Let $X$ be a space, $\kappa$ a cardinal, and $\scr{C}$ a cover of $X$ such that $|\scr{C}|\leq\kappa$ and for all $C\in\scr{C}$, $t(C)\leq\kappa$ and $C$ is dense in $X$. If $\pi\chi(X)\leq\kappa$ then $at(X)\leq\kappa$.
\end{proposition}

\begin{proof}
It only needs to be shown that $X=cl_\kappa(C)$ for all $C\in\scr{C}$. Let $C\in\scr{C}$ and let $x\in X$. Let $\scr{B}$ be a local $\pi$-base at $x$ such that $|\scr{B}\leq\kappa$. As $C$ is dense in $X$, for all $B\in\scr{B}$ there exists $x_B\in B\meet C$. Then $x\in\closure{\{x_B:B\in\scr{B}\}}$ and $x\in cl_\kappa(C)$.
\end{proof}

The following is a consequence of Corollary~\ref{pichisigma} and Proposition~\ref{atcondition}.

\begin{corollary}\label{almostct}
Let $X$ be a power homogeneous compactum and with a countable cover consisting of dense, countably tight subspaces. Then $X$ is almost countably tight.
\end{corollary}

The following was proved in \cite{Car2018}. It represents an extension of Theorem 4.1 in \cite{JVM2018}. We see that it is a consequence of Theorem~\ref{cptPH}.

\begin{theorem}[Theorem 4.6 \cite{Car2018}]\label{previousPH}
Let $X$ be a power homogeneous compactum and suppose there exists a countable cover of $X$ consisting of dense, countably tight subspaces. Then $|X|\leq\mathfrak{c}$.
\end{theorem}

\begin{proof}
As $X$ is a $\sigma$-CT power homogeneous compactum, by Corollary~\ref{pichisigma} it follows that $\pi\chi(X)\leq\omega$. Now use Corollary~\ref{almostct} to conclude that $X$ is almost countably tight. As $X$ is a power homogeneous compactum, by Theorem~\ref{cptPH}, we have $|X|\leq 2^{at(X)\pi\chi(X)}\leq 2^{\omega\cdot\omega}=\mathfrak{c}$.
\end{proof}

\section{Generalizations involving the weak tightness.}

In this section we extend extend Theorem~\ref{cpthomog} (parts (a) and (b)) to the Hausdorff and regular settings. We will need Theorem~\ref{compactsubset2}, which applies to Hausdorff spaces $X$ with the property $wt(X)pct(X)\leq\kappa$ for a cardinal $\kappa$. We will also need the following lemma from \cite{Car2018}.

\begin{lemma}[Lemma 3.4 in \cite{Car2018}]\label{subcover}
Let $X$ be a space, $D\sse X$, $wt(X)\leq\kappa$, and let $\scr{G}$ be a cover of $\closure{D}$ consisting of $G^c_\kappa$-sets of $X$. Then there exists $\scr{G}^\prime\sse\scr{G}$ such that $|\scr{G}^\prime|\leq |D|^\kappa$ and $\scr{G}^\prime$ covers $D$.
\end{lemma}

The following was defined by Bella and Spadaro in \cite{BS2020}.

\begin{definition}[\cite{BS2020}]\label{pwL}
Let $X$ be a space. The \emph{piecewise weak Lindel\"of degree} $pwL(X)$ of $X$ is the least infinite cardinal $\kappa$ such that for every open cover $\scr{U}$ of $X$ and every decomposition $\{\scr{U}_i:i\in I\}$ of $\scr{U}$, there are families $\scr{V}_i\in[\scr{U}_i]^{\leq\kappa}$ for every $i\in I$ such that $X\sse\Un\{\closure{\Un\scr{V}_i}:i\in I\}$. The \emph{piecewise weak Lindel\"of degree for closed sets} $pwL_c(X)$ of $X$ is the least infinite cardinal $\kappa$ such that for every closed set $F\sse X$, for every open cover $\scr{U}$ of $F$, and every decomposition $\{\scr{U}_i:i\in I\}$ of $\scr{U}$, there are families $\scr{V}_i\in[\scr{U}_i]^{\leq\kappa}$ for every $i\in I$ such that $F\sse\Un\{\closure{\Un\scr{V}_i}:i\in I\}$. 
\end{definition}

It is clear that $pwL_c(X)\leq L(X)$ and it was shown in \cite{BS2020} that $pwL_c(X)\leq c(X)$. It was also shown in \cite{BS2020} that if $X$ is Hausdorff then $|X|\leq 2^{pwL_c(X)\chi(X)}$, providing a common proof of the well-known cardinality bounds $2^{L(X)\chi(X)}$ and $2^{c(X)\chi(X)}$ for Hausdorff spaces.

The following improves Theorem 3.5 in \cite{BC2020b} by replacing $t(X)$ with $wt(X)$. Unlike Theorem 3.5 in \cite{BC2020b}, the closing-off argument uses an increasing chain of $\scr{C}$-saturated sets.

\begin{theorem}\label{pwL_c}
Let $X$ be a Hausdorff space and $\kappa$ a cardinal such that $pwL_c(X)wt(X)\leq\kappa$. Let $\scr{G}$ be a cover of $X$ consisting of compact sets $G$ such that $\chi(G,X)\leq\kappa$. Then there exists $\scr{G}^\prime\sse\scr{G}$ such that $X=\closure{\Un\scr{G}^\prime}$ and $|\scr{G}^\prime|\leq 2^\kappa$.
\end{theorem}

\begin{proof}
Let $\scr{C}$ be a cover of $X$ witnessing that $wt(X)\leq\kappa$. For all $G\in\scr{G}$, let $\scr{U(G)}$ be a neighborhood base at $G$ such that $|\scr{U}(G)|\leq\kappa$. If $\scr{G}^\prime\sse\scr{G}$, let $\scr{U}(\scr{G}^\prime)=\Un\{\scr{U}(G):G\in\scr{G}^\prime\}$. Note each $G\in\scr{G}$ us a $G^c\kappa$-set as each $G\in\scr{G}$ is compact, $X$ is Hausdorff, and $\chi(G,X)\leq\kappa$. We build an increasing chain $\{A_\alpha:\alpha<\kappa^+\}$ of $\scr{C}$-saturated subsets of $X$ and an increasing chain $\{\scr{G}_\alpha:\alpha<\kappa^+\}$ of subsets of $\scr{G}$ such that
\begin{enumerate}
\item $|\scr{G}_\alpha|\leq 2^\kappa$ and $|A_\alpha|\leq 2^\kappa$,
\item $\closure{A_\alpha}\sse\Un\scr{G}_\alpha$ and each $A_\alpha$ is $\scr{C}$-saturated,
\item Whenever $\{V_\alpha:\beta<\kappa\}\in\left[\left[\scr{U}(\scr{G}_\alpha)\right]^{\leq\kappa}\right]^{\leq\kappa}$ for some $\alpha<\kappa^+$ and $X\minus\Un\{\closure{\Un\scr{V}_\beta}:\beta<\kappa\}\neq\es$, then $A_{\alpha+1}\minus\Un\{\closure{\Un\scr{V}_\beta}:\beta<\kappa\}\neq\es$.
\end{enumerate}
Let $A_0$ be an arbitrary $\scr{C}$-saturated set of cardinality at most $2^\kappa$. (For example, let $A_0=S(x)$ for some $x\in X$). For limit ordinals $\alpha<\kappa^+$, let $A_\alpha=\Un_{\beta<\alpha}A_\beta$. Then $|A_\alpha|\leq 2^\kappa$, and $A_\alpha$ is $\scr{C}$-saturated as is a union of $\scr{C}$-saturated sets is $\scr{C}$-saturated. By applying Lemma~\ref{subcover} to $\closure{A_\alpha}$, we get $\scr{G}_\alpha$ as required.

For successor ordinals $\alpha+1$, we ensure (3) above is satisfied. For all $\scr{V}=\{\scr{V}_\beta:\beta<\kappa\}\in\left[\left[\scr{U}(\scr{G}_\alpha)\right]^{\leq\kappa}\right]^{\leq\kappa}$ such that $X\minus\Un\{\closure{\Un\scr{V}_\beta}:\beta<\kappa\}\neq\es$, let $x_\scr{V}\in X\minus\Un\{\closure{\Un\scr{V}_\beta}:\beta<\kappa\}$. Define
\begin{align}
A_{\alpha+1}&=A_\alpha\un\{S(x_{\scr{V}}):\scr{V}=\{\scr{V}_\beta:\beta<\kappa\}\in\left[\left[\scr{U}(\scr{G}_\alpha)\right]^{\leq\kappa}\right]^{\leq\kappa}\textup{ and }\notag\\
&X\minus\Un\{\closure{\Un\scr{V}_\beta}:\beta<\kappa\}\neq\es\}\notag
\end{align}

As $|S(x_\scr{V})|\leq 2^\kappa$ and $\left|\left[\left[\scr{U}(\scr{G}_\alpha)\right]^{\leq\kappa}\right]^{\leq\kappa}\right|\leq 2^\kappa$, we have $|A_{\alpha+1}|\leq 2^\kappa$. As each $S(x_\scr{V})$ is $\scr{C}$-saturated, it follows that $A_{\alpha+1}$ is $\scr{C}$-saturated. Apply Lemma~\ref{subcover} again to $\closure{A_{\alpha+1}}$ to obtain $\scr{G}_{\alpha+1}$.

Let $A=\Un\{A_\alpha:\alpha<\kappa^+\}$. By Lemma~\ref{closureunion}, it follows that $\closure{A}=\Un\{\closure{A_\alpha}:\alpha<\kappa^+\}$. Let $\scr{G}^\prime=\Un\{\scr{G}_\alpha:\alpha<\kappa^+\}$ and note that $\closure{A}\sse\Un\scr{G}^\prime$. We show $X=\closure{\Un\scr{G}^\prime}$. Suppose by way of contradiction that $X\minus\closure{\Un\scr{G}^\prime}\neq\es$ and let $x\in X\minus\closure{\Un\scr{G}^\prime}$. Now there exists $G\in\scr{G}$ such that $x\in G$, as $\scr{G}$ is a cover of $X$. Since $G$ is compact and $X$ is Hausdorff, it follows that $G$ is regular. Note that $x\in G\minus\closure{\Un\scr{G}^\prime}$ and that $G\minus\closure{\Un\scr{G}^\prime}$ is an open set in $G$. By regularity of $G$ there exists a set $H\sse G$ such that $H$ is a closed $G_\delta^c$-set in $G$ and $x\in H\sse G\minus\closure{\Un\scr{G}^\prime}$. Observe that $H$ is compact as it is a closed subset of the compact set $G$. Thus $\omega=\psi(H,G)=\chi(H,G)$. Then, by Lemma~\ref{char}, $\chi(H,X)\leq\chi(H,G)\cdot\chi(G,X)\leq\omega\cdot\kappa=\kappa$ because $H$ and $G$ are compact.

Let $\scr{N}$ be a neighborhood base at $H$ in $X$ such that $|\scr{N}|\leq\kappa$. Now, for all $G\in\scr{G}^\prime$ we have $H\meet G=\es$. Since $H$ and $G$ are compact and $X$ is Hausdorff, there exists $N_G\in\scr{N}$ and $U(G)\in\scr{U}(G)$ such that $N_G\meet U(G)=\es$. Let $\scr{U}=\{U(G):G\in\scr{G}^\prime\}$ and note $\scr{U}$ is a cover of $\closure{A}$. For all $N\in\scr{N}$ let $\scr{U}_N=\{U\in\scr{U}:U\meet N=\es\}$ and observe that $x\in H\sse X\minus\closure{\Un\scr{U}_N}$. The family $\{\scr{U}_N:N\in\scr{N}\}$ is a decomposition of $\scr{U}$. Since $pwL_c(X)\leq\kappa$, for all $N\in\scr{N}$ there exists $\scr{V}_N\in[\scr{U}_N]^{\leq\kappa}$ such that $\closure{A}\sse\Un\{\closure{\Un\scr{V}_N}:N\in\scr{N}\}$. Then $x\in H\sse X\minus\closure{\Un\scr{V}_N}$ for each $N\in\scr{N}$ and $x\in X\minus\Un\{\closure{\Un\scr{V}_N}:N\in\scr{N}\}\neq\es$. As $\{\scr{V}_N:N\in\scr{N}\}$ fits condition (3) above and $\{\scr{V}_N:N\in\scr{N}\}\in\left[\left[\scr{U}(\scr{G}_\alpha)\right]^{\leq\kappa}\right]^{\leq\kappa}$ for some $\alpha<\kappa^+$, we have 
$$\es\neq A_{\alpha+1}\minus\Un\{\closure{\Un\scr{V}_N}:N\in\scr{N}\}\sse A\sse \closure{A}\sse\Un\{\closure{\Un\scr{V}_N}:N\in\scr{N}\},$$
a contradiction. Thus $X=\closure{\Un\scr{G}^\prime}$ which completes the proof.
\end{proof}

If $X$ is homogeneous the above theorem is used to obtain a bound for the cardinality of $X$.

\begin{theorem}\label{pbound}
If $X$ is homogeneous and Hausdorff then $|X|\leq 2^{pwL_c(X)wt(X)\pi\chi(X)pct(X)}$.
\end{theorem}

\begin{proof}
Let $\kappa=pwL_c(X)wt(X)\pi\chi(X)pct(X)$, and let $\scr{C}$ be a cover of $X$ witnessing that $wt(X)\leq\kappa$. As $wt(X)pct(X)\leq\kappa$, by Theorem~\ref{compactsubset2} there exists a nonempty compact set $G\sse X$ and $H\sse X$ such that $|H|\leq 2^\kappa$, $G\sse\closure{H}$, $\chi(G,X)\leq\kappa$, and $H$ is $\scr{C}$-saturated.

Fix $p\in X$. For all $x\in X$ there exists a homeomorphism $h_x:X\to X$ such that $h_x(p)=X$. Then $\scr{G}^\prime=\{h_x[G]:x\in X\}$ is a cover of $X$ by compact sets of character at most $\kappa$. By Theorem~\ref{pwL_c} above there exists a subfamily $\scr{G}\sse\scr{G}^\prime$ such that $|\scr{G}|\leq 2^\kappa$ and $X=\closure{\Un\scr{G}}$. Also, for all $G\in\scr{G}$ there exists $H_G\in[X]^{\leq 2^\kappa}$ such that $G\sse\closure{H_G}$. It follows that $D=\Un_{G\in\scr{G}}H_G$ is dense in $X$. Since $|D|\leq 2^\kappa\cdot 2^\kappa=2^\kappa$, we have $d(X)\leq 2^\kappa$. Now apply the fact that $|X|\leq d(X)^{\pi\chi(X)}$ for any Hausdorff homogeneous space (Lemma~\ref{ridPH}).
\end{proof}

Despite the fact that four cardinal invariants are used in the bound in Theorem~\ref{pbound}, each of these invariants should be regarded as ``small'' in a rough sense. In addition, as $\pi\chi(X)\leq t(X)pct(X)$ for a Hausdorff space $X$ and $wt(X)\leq t(X)$, we see that Theorem~\ref{pbound} improves Corollary 3.7 in \cite{BC2020b}. This latter result states that if $X$ is homogeneous and Hausdorff, then $|X|\leq 2^{pwL_c(X)t(X)pct(X)}$. Theorem~\ref{pbound} is also an extension of Theorem~\ref{cpthomog}(b).

It was shown in \cite{CR2008} that the cardinality of a homogeneous Hausdorff space $X$ is at most $2^{c(X)\pi\chi(X)}$. Observe that the above bound is a variation of this, as $pwL_c(X)\leq c(X)$ and $wt(X)$ and $pct(X)$ are added to the bound in Theorem~\ref{pbound}.

\begin{lemma}[Lemma 3.4.8 in \cite{Ridderbos2007}]\label{nw2w}
Let $X$ be a space, and suppose $pct(X)\leq\kappa$ and $nw(X)\leq 2^\kappa$. Then $w(X)\leq 2^\kappa$.
\end{lemma}

\begin{corollary}\label{anotherw(X)}
If $X$ is homogeneous and Hausdorff then $w(X)\leq 2^{pwL_c(X)wt(X)\pi\chi(X)pct(X)}$.
\end{corollary}

\begin{proof}
Let $\kappa=pwL_c(X)wt(X)\pi\chi(X)pct(X)$. By Theorem~\ref{pbound}, we have $nw(X)\leq |X|\leq 2^\kappa$. As $pct(X)\leq\kappa$, it follows by Lemma~\ref{nw2w} that $w(X)\leq 2^\kappa$.
\end{proof}

\begin{lemma}[Proposition 3.5 in \cite{BC2020}]\label{Csatnetwork}
Let $X$ be a regular space, $\kappa=wt(X)$, and $D\sse X$ be $\scr{C}$-saturated. Then $nw(\closure{D})\leq |D|^\kappa$.
\end{lemma}

We arrive at a result for regular, homogeneous spaces. It extends Theorem~\ref{cpthomog}(a), which states that $w(X)\leq 2^{wt(X)}$ for homogeneous compacta. Compare to Corollary~\ref{anotherw(X)}.

\begin{theorem}\label{w(X)}
Let $X$ be a homogeneous regular Hausdorff space. Then $w(X)\leq 2^{L(X)wt(X)pct(X)}$.
\end{theorem}

\begin{proof}
Let $\kappa=L(X)wt(X)pct(X)$, and let $\scr{C}$ be a cover of $X$ witnessing that $wt(X)\leq\kappa$. By Theorem~\ref{compactsubset2}, there exists a nonempty compact set $G\sse X$ and $H\sse X$ such that $H$ is $\scr{C}$-saturated, $|H|\leq 2^\kappa$, $G\sse\closure{H}$, and $\chi(G,X)\leq\kappa$. Fix $p\in G$. As $X$ is homogeneous, for all $x\in X$ there exists a homeomorphism $h_x:X\to X$ such that $h_x(p)=x$. Then $\scr{F}=\{h_x[G]:x\in X\}$ is a cover of $X$ by compact $G^c_\kappa$-sets  and, for all $x\in X$, $h_x[G]\sse h_x[\closure{H}]$. By Theorem~\ref{L(X)}, there exists $A\sse X$ such that $|A|\leq 2^\kappa$ and $X=\Un\{h_x[G]:x\in A\}=\Un\{h_x[\closure{H}]:x\in A\}$.

As $H$ is $\scr{C}$-saturated and $X$ is regular, by Lemma~\ref{Csatnetwork} it follows that $nw(\closure{H})\leq|H|^\kappa\leq (2^\kappa)^\kappa=2^\kappa$. Let $\scr{N}$ be a network for $\closure{H}$ such that $|\scr{N}|\leq 2^\kappa$. Let $\scr{M}=\{h_x[N]:N\in\scr{N}, x\in A\}$ and note $|\scr{M}|\leq 2^\kappa\cdot 2^\kappa=2^\kappa$. We show $\scr{M}$ is a network for $X$. Let $y\in U$ where $U$ is open in $X$. There exists $x\in A$ such that $y\in h_x[\closure{H}]\meet U$. Then $h_x^{\leftarrow}(y)\in \closure{H}\meet h_x^{\leftarrow}[U]$. As $\scr{N}$ is a network for $\closure{H}$, there exists $N\in\scr{N}$ such that $h_x^{\leftarrow}(y)\in N\sse\closure{H}\meet h_x^{\leftarrow}[U]$. It follows that $y\in h_x[N]\sse h_x[\closure{H}]\meet U\sse U$. As $h_x[N]\in\scr{M}$, we have that $\scr{M}$ is a network for $X$ and $nw(X)\leq 2^\kappa$. Therefore, as $pct(X)\leq\kappa$, it follows by Lemma~\ref{nw2w} that $w(X)\leq 2^\kappa$. This completes the proof.
\end{proof}

We do not know if Theorem~\ref{w(X)} holds when the homogeneous condition is replaced with the weaker power homogeneous property.

\begin{question}
If $X$ is a power homogeneous regular Hausdorff space, is $w(X)\leq 2^{L(X)wt(X)pct(X)}$?
\end{question}

\begin{question} If $X$ is power homogeneous and Hausdorff, is $|X|\leq 2^{pwL_c(X)wt(X)\pi\chi(X)pct(X)}$?
\end{question}

Finally, as Theorem~\ref{cptPH} only gives a partial answer to Question 3.13 in \cite{BS2020}, we ask that question again.

\begin{question} If $X$ is a power homogeneous compactum, is $|X|\leq 2^{wt(X)\pi\chi(X)}$?
\end{question}

\end{document}